\documentclass[11pt]{amsart}
\usepackage{amsfonts}
\usepackage{mathrsfs}
\usepackage{amsmath}
\usepackage{amsmath, amsthm, amssymb}
\usepackage{color}

\renewcommand{\(}{\left( }
\renewcommand{\)}{\right) }

\renewcommand{\theequation}{\theequation. \arabic{equation}}
\numberwithin{equation}{section}
\newtheorem{thm}{Theorem}[section]

\newtheorem{rem}[thm]{Remark}
\newtheorem{prop}[thm]{Proposition}

\def\squarebox#1{\hbox to #1{\hfill\vbox to #1{\vfill}}}

\begin{document}
\title[ON THE $q$-DERIVATIVE AND $q$-SERIES EXPANSIONS]
{ON THE $q$-DERIVATIVE AND $q$-SERIES EXPANSIONS}
\author{Zhi-Guo Liu}
\date{\today}
\address{ Department of Mathematics, East China Normal University, 500 Dongchuan Road,
Shanghai 200241, P. R. China} \email{zgliu@math.ecnu.edu.cn;
liuzg@hotmail.com}
\thanks{The  author was supported in part by
the National Science Foundation of China}
\thanks{ 2010 Mathematics Subject Classifications :  05A30,
33D15, 11E25.}
\thanks{ Keywords: $q$-Series; $q$-derivative; Hecke-type series; sums of squares.}
\begin{abstract}
Using a general $q$-series expansion, we derive some nontrivial $q$-formulas involving many infinite products. A multitude of Hecke--type series identities are derived. Some general formulas for sums of any number of squares are given. A new representation for the generating function for sums of three triangular numbers is derived, which is slightly different from that of Andrews, also implies the famous result of Gauss where every integer is the sum of three triangular numbers.
\end{abstract}
\maketitle
\section{Introduction}

Throughout the paper, we use the standard $q$-notations.
For $0<q<1$, we define the $q$-shifted factorials as
\begin{equation*}
(a; q)_0=1,\quad (a; q)_n=\prod_{k=0}^{n-1}(1-aq^k), \quad (a;
q)_\infty=\prod_{k=0}^\infty (1-aq^k);
\end{equation*}
and for convenience, we also adopt the following compact notation for the multiple
$q$-shifted factorial:
\begin{equation*}
(a_1, a_2,...,a_m;q)_n=(a_1;q)_n(a_2;q)_n ... (a_m;q)_n,
\end{equation*}
where $n$ is an integer or $\infty$.

The basic hypergeometric series
${_r\phi_s}$ is defined as
\begin{equation*}
{_r\phi_s} \left({{a_1, a_2, ..., a_{r}} \atop {b_1, b_2, ...,
b_s}} ;  q, z  \right) =\sum_{n=0}^\infty \frac{(a_1, a_2, ...,
a_{r};q)_n} {(q,  b_1, b_2, ..., b_s ;q)_n}\left((-1)^n q^{n(n-1)/2}\right)^{1+s-r} z^n.
\end{equation*}

For any function $f(x)$, the  $q$-derivative of $f(x)$
with respect to $x,$ is defined as
\begin{equation*}
\mathcal{D}_{q,x}\{f(x)\}=\frac{f(x)-f(qx)}{x},
\end{equation*}
and we further define  $\mathcal{D}_{q,x}^{0} \{f\}=f,$ and
for $n\ge 1$, $\mathcal{D}_{q, x}^n \{f\}=\mathcal{D}_{q, x}\{\mathcal{D}_{q, x}^{n-1}\{f\}\}.$

Using some basic properties of  the $q$-derivative, we \cite{Liu} prove the following $q$-expansion formula.
\begin{thm} {\rm (Liu)}\label{liuthm1} If $f(x)$ is an analytic function near $x=0$, then,  we have
\begin{equation*}
f(a)=\sum_{n=0}^\infty \frac{(1-\alpha q^{2n})(\alpha q/a; q)_n a^n}{(q, a;
q)_n}\left[ \mathcal{D}^n_{q, x}\{f(x)(x; q)_{n-1}\} \right]_{x=\alpha q}.
\end{equation*}
\end{thm}
Using Theorem \ref{liuthm1},  we \cite{Liu2013} established
the following theorem.
\begin{thm} {\rm (Liu)}\label{liuthm2}  If $f(x)$ is an analytic function near $x=0,$
then,  under suitable convergence conditions,  we have
\begin{align*}
&\frac{(\alpha q, \alpha ab/q; q)_\infty}
{(\alpha a, \alpha b; q)_\infty} f(\alpha a)\\
=\sum_{n=0}^\infty & \frac{(1-\alpha q^{2n}) (\alpha, q/a; q)_n (a/q)^n}
{(1-\alpha
)(q, \alpha a; q)_n}
\sum_{k=0}^n \frac{(q^{-n}, \alpha q^n; q)_k q^k}
{(q, \alpha b; q)_k}f(\alpha q^{k+1}).\nonumber
\end{align*}
\end{thm}
Several interesting applications of this formula are discussed in \cite{Liu2013}.
In particular, this formula leads to a new proof of the orthogonality relation
for the Askey-Wilson polynomials.

In this paper we continue to discuss the applications of Theorems \ref{liuthm2}.
Many nontrivial $q$-formulae are derived.  In particular,
we prove the following remarkable $q$-formula involving many infinite products.
\begin{thm} \label{liunewthma} If
$
\max\{|\alpha a|, |\alpha b|, |\alpha b_1|, |\alpha ac_1/q|, \cdots |\alpha b_m|, |\alpha a c_m/q|\}<1,
$ and  $m, l$ are two nonnegative integers,  then,  we have the $q$-formula
 \begin{align*}
& (a/q)^l \frac{(\alpha q, \alpha ab/q; q)_\infty}
{(\alpha a, \alpha b; q)_\infty} \prod_{j=1}^m \frac{(\alpha a b_j/q, \alpha c_j; q)_\infty}
{(\alpha a c_j/q,  \alpha b_j; q)_\infty}\\
&=\sum_{n=0}^\infty  \frac{(1-\alpha q^{2n}) (\alpha, q/a; q)_n (a/q)^n}
{(1-\alpha)(q, \alpha a; q)_n}
{_{m+2}\phi_{m+1}}\left( {{q^{-n}, \alpha q^n, \alpha c_1, \cdots,  \alpha c_m}
\atop{\alpha b, \alpha b_1, \cdots, \alpha b_m}}; q, q^{l+1}\right).
 \end{align*}
\end{thm}
\begin{proof}
It is easily seen that the following function is analytic near $x=0$:
\[
x^l\prod_{j=1}^m \frac{(b_jx/q; q)_\infty}{(c_jx/q; q)_\infty}.
\]
Thus, we can replace  $f(x)$ by this function  in Theorem \ref{liuthm2}. By a direct
computation, we immediately find that
\[
f(\alpha a)=(\alpha a)^l\prod_{j=1}^m \frac{(\alpha a b_j/q; q)_\infty}{(\alpha a c_j/q; q)_\infty},
\]
\[
f(\alpha q^{k+1})=(\alpha q^{k+1})^l\prod_{j=1}^m \frac{(\alpha b_j q^k; q)_\infty}{(\alpha c_j q^{k}; q)_\infty}.
\]
Substituting these two equations into Theorem \ref{liuthm2} and simplifying, we complete the
proof of Theorem \ref{liunewthma}.
\end{proof}
When $l=0,$ the above theorem reduces to the following $q$-identity.
\begin{thm} \label{liunewthmb} If
$
\max\{|\alpha a|, |\alpha b|, |\alpha b_1|, |\alpha ac_1/q|, \cdots |\alpha b_m|, |\alpha a c_m/q|\}<1,
$  then,  we have
 \begin{align*}
&  \frac{(\alpha q, \alpha ab/q; q)_\infty}
{(\alpha a, \alpha b; q)_\infty} \prod_{j=1}^m \frac{(\alpha a b_j/q, \alpha c_j; q)_\infty}
{(\alpha a c_j/q, \alpha b_j; q)_\infty}\\
&=\sum_{n=0}^\infty  \frac{(1-\alpha q^{2n}) (\alpha, q/a; q)_n (a/q)^n}
{(1-\alpha)(q, \alpha a; q)_n}
{_{m+2}\phi_{m+1}}\left( {{q^{-n}, \alpha q^n, \alpha c_1, \cdots,  \alpha c_m}
\atop{\alpha b, \alpha b_1, \cdots, \alpha b_m}}; q, q\right).
 \end{align*}
 \end{thm}
 When $a=0,$ Theorem \ref{liunewthmb} immediately becomes the following interesting formula.
 \begin{thm} \label{liunewthmc} If
$
\max\{|\alpha b|, |\alpha b_1|, \cdots |\alpha b_m|\}<1,
$  then,  we have
 \begin{align*}
&  \frac{(\alpha q; q)_\infty}
{(\alpha b; q)_\infty} \prod_{j=1}^m \frac{(\alpha c_j; q)_\infty}
{(\alpha b_j; q)_\infty}\\
&=\sum_{n=0}^\infty  \frac{(1-\alpha q^{2n}) (\alpha; q)_n (-1)^n q^{n(n-1)/2}}
{(1-\alpha)(q; q)_n}
{_{m+2}\phi_{m+1}}\left( {{q^{-n}, \alpha q^n, \alpha c_1, \cdots,  \alpha c_m}
\atop{\alpha b, \alpha b_1, \cdots, \alpha b_m}}; q, q\right).
 \end{align*}
 \end{thm}
 Setting $b=b_1=\cdots=b_m=0$ and $c_1=c_2=\cdots=c_m=0,$ respectively, in Theorem~\ref{liunewthmc},
 we obtain the following two theorems.
 \begin{thm}\label{liunewthmd}We have the $q$-summation formula
  \begin{align*}
 (\alpha q; q)_\infty \prod_{j=1}^m (\alpha c_j; q)_\infty&=\sum_{n=0}^\infty  \frac{(1-\alpha q^{2n}) (\alpha; q)_n (-1)^n q^{n(n-1)/2}}
{(1-\alpha)(q; q)_n}\\
&\quad \times {_{m+2}\phi_{m+1}}\left( {{q^{-n}, \alpha q^n, \alpha c_1, \cdots,  \alpha c_m}
\atop{0, 0, \cdots, 0}}; q, q\right).
 \end{align*}
 \end{thm}
  \begin{thm} \label{liunewthme} If
$\max\{|\alpha b|, |\alpha b_1|, \cdots |\alpha b_m|\}<1, $  then,  we have
 \begin{align*}
\frac{(\alpha q; q)_\infty}{(\alpha b; q)_\infty \prod_{j=1}^m (\alpha b_j; q)_\infty}&=\sum_{n=0}^\infty  \frac{(1-\alpha q^{2n}) (\alpha; q)_n (-1)^n q^{n(n-1)/2}}
{(1-\alpha)(q; q)_n}\\
&\quad \times{_{m+2}\phi_{m+1}}\left( {{q^{-n}, \alpha q^n, 0, 0, \cdots,  0}
\atop{\alpha b, \alpha b_1, \cdots, \alpha b_m}}; q, q\right).
 \end{align*}
 \end{thm}
 Many applications of Theorems~\ref{liunewthmb}, \ref{liunewthmc}, \ref{liunewthmd}, \ref{liunewthme} are discussed
 in this paper. For example, we prove the following rather striking identities:
  \begin{equation*}
 \(\sum_{n=-\infty}^\infty (-1)^n q^{n^2}\)^{m+2}=1+2\sum_{n=1}^\infty (-1)^n {_{m+2}\phi_{m+1}}\left( {{q^{-n},  q^n, q, \cdots,  q}
\atop{-q, 0, \cdots, 0}}; q, q\right),
 \end{equation*}
 \begin{equation*}
 \prod_{j=1}^\infty \frac{1}{(1-q^j)^{m}}=\sum_{n=0}^\infty (-1)^n q^{n(n-1)/2}{_{m+2}\phi_{m+1}}\left( {{q^{-n},  q^{n+1}, 0, \cdots, 0}
\atop{q, q, \cdots, q}}; q, q\right),
\end{equation*}
\begin{equation*}
 \prod_{j=1}^\infty (1-q^j)^{m+1}=\sum_{n=0}^\infty (-1)^n q^{n(n-1)/2}{_{m+2}\phi_{m+1}}\left( {{q^{-n},  q^{n+1}, q, \cdots, q}
\atop{0, 0, \cdots, 0}}; q, q\right).
\end{equation*}
\begin{equation*}
\(\sum_{n=-\infty}^\infty q^{n^2}\)\(\sum_{n=0}^\infty q^{n(n+1)/2}\)=\sum_{n=0}^\infty \sum_{j=-n}^n (-1)^{n+j}(1+q^{2n+1}) q^{2n^2+n-2j^2}.
\end{equation*}
It is obvious that the first identity of the above gives a new method to compute the number of representations of a positive integer as sums of any number of squares.

Let $p(n)$ denote the number of ways in which $n$ can be written as a sum of positive integers. Then it is well-known that
\[
\sum_{n=0}^\infty p(n)q^n =\prod_{j=1}^\infty \(\frac{1}{1-q^j}\).
\]

Thus the case $m=1$ of the second identity of the above gives a new identity involving the partition:
\[
\sum_{n=0}^\infty p(n)q^n=\sum_{n=0}^\infty (-1)^n q^{n(n-1)/2}{_{3}\phi_{2}}\left( {{q^{-n},  q^{n+1}, 0}
\atop{q, q}}; q, q\right).
\]

We also prove the following identity, which is similar to Andrews' identity \cite[Eq. (1.5)]{Andrews86} for sums of three triangular numbers,
which also implies Gauss' famous result that every integer is the sum of three triangular numbers:
\begin{equation*}
\(\sum_{j=0}^\infty q^{j(j+1)/2}\)^3=\sum_{n=0}^\infty \sum_{j=-n}^n \(\frac{1+q^{2n+1}}{1-q^{2n+1}}\) q^{2n^2+2n-2j^2-j}.
\end{equation*}

Taking $m=1, b_1=c, c_1=bc/q $ in Theorem~\ref{liunewthmb} and using the $q$-Pfaff-Saalsch\"utz summation formula
in the resulting equation, we can obtain the following theorem. It should be pointed out that
there is a misprint in \cite[Theorem~1.3]{Liu2013}.
\begin{thm} \label{rogersthm} {\rm(Rogers' $_6\phi_5$ summation)}For $|\alpha abc/q^2|<1,$ we have
\begin{align*}
{_6 \phi_5} \left({{\alpha, q\sqrt{\alpha}, -q\sqrt{\alpha}, q/a, q/b, q/c}
\atop{\sqrt{\alpha}, -\sqrt{\alpha},\alpha a, \alpha b, \alpha c}}; q, \frac{\alpha abc}{q^2}\right)
 =\frac{(\alpha q, \alpha ab/q, \alpha ac/q, \alpha bc/q; q)_\infty}
{(\alpha a, \alpha b, \alpha c, \alpha abc/q^2; q)_\infty}. \nonumber
\end{align*}
\end{thm}

The paper is organized as follows. Some limiting cases of Watson's $q$-analog of Whipple's theorem are discussed in
 Section~2.   Some applications of Theorems~\ref{liunewthmb}, \ref{liunewthmc}, \ref{liunewthmd}, \ref{liunewthme} are discussed in
 Section~3.

 Section~4 is devoted to Hecke-type identities. For example, we prove the following theorem.
\begin{thm}\label{liunewthmf} For $|ab|<1,$ we have the $q$-formula
\begin{align*}
&\frac{(q, ab; q)_\infty}{(qa, qb; q)_\infty}\sum_{n=0}^\infty \frac{(q/a, q/b; q)_n (-ab)^n}{(q^2; q^2)_n}\\
&\quad=\sum_{n=0}^\infty \sum_{j=-n}^n (-1)^j \frac{(1-q^{2n+1})(q/a, q/b; q)_n (ab)^n q^{n^2-j^2}}{(qa, qb; q)_n}.
\end{align*}
\end{thm}
When $a=1$ and $b=0,$  Theorem \ref{liunewthmf} immediately reduces to the following well-known identity
due to Andrews, Dyson and Hickerson \cite{ADH1988}.
\begin{prop}\label{aliupp1} {(\rm Andrews, Dyson and Hickerson )} We have
\[
\sum_{n=0}^\infty \frac{q^{n(n+1)/2}}{(-q; q)_n}
=\sum_{n=0}^\infty \sum_{j=-n}^n (-1)^{n+j} (1-q^{2n+1})q^{(3n^2+n)/2-j^2}.
\]
\end{prop}
Setting $(a, b)=(q^{1/2}, -q^{1/2}), (0, 0)$ and $(-1, 0)$, respectively, in Theorem \ref{liunewthmf}, we are
led to the following proposition.
\begin{prop}\label{aliupp2} We have the Hecke-type series identities
\begin{align*}
&\sum_{n=0}^\infty \frac{(q; q^2)_n q^n}{(q^2; q^2)_n}
=\frac{(q; q^2)_\infty}{(q^2; q^2)_\infty}
\sum_{n=0}^\infty \sum_{j=-n}^n (-1)^{n+j} q^{n^2+n-j^2},\\
&\sum_{n=0}^\infty \frac{(-1)^n q^{n^2+n}}{(q^2; q^2)_n}
=\frac{1}{(q; q)_\infty} \sum_{n=0}^\infty \sum_{j=-n}^n (1-q^{2n+1}) (-1)^j q^{2n^2+n-j^2},\\
&\sum_{n=0}^\infty \frac{(-1)^n q^{n(n+1)/2}}{(q; q)_n}
=\frac{(-q; q)_\infty}{(q; q)_\infty}
\sum_{n=0}^\infty \sum_{j=-n}^n (-1)^j (1-q^{2n+1}) q^{(3n^2+n)/2-j^2}.
\end{align*}
\end{prop}
In Section~5, we prove the following theorem by using the Sears $_4\phi_3$ transformation.
\begin{thm}\label{newliuthmg} For $|{\alpha ab }/{q}|<1,$ we have the $q$-formula
\begin{align*}
&\frac{(\alpha q, \alpha ab/q; q)_\infty}
{(\alpha a, \alpha b; q)_\infty} {_3\phi_2} \left({{q/a, q/b, \beta} \atop { c, d}} ;  q, \frac{\alpha ab }{q} \right) \\
&=\sum_{n=0}^\infty \frac{(1-\alpha q^{2n}) (\alpha, q/a, q/b, q\alpha/c; q)_n (\alpha abc)^n q^{n^2-2n}}
{(1-\alpha)(q, \alpha a, \alpha b, c; q)_n}\\
&\qquad \qquad \times{_3\phi_2} \left({{q^{-n}, \alpha q^n, d/\beta} \atop { d, q\alpha/c}} ;  q, \frac{q\beta}{c} \right).
\end{align*}
\end{thm}
By letting $a=b=d=0$ and $c=-q$ and replacing $\beta$ by $-q\beta$,
we get the following formula due to Andrews (see,  for example,  \cite[Theorem~1]{Andrews2012}).
\begin{prop}\label{aliupp3} {\rm (Andrews) } We have the $q$-identity
\begin{align*}
&\sum_{n=0}^\infty \frac{q^{n^2} (-q\beta; q )_n \alpha^n}{(q^2; q^2)_n}\\
&=\frac{1}{(q\alpha; q)_\infty}
\sum_{n=0}^\infty \frac{(1-\alpha q^{2n}) (\alpha^2; q^2)_n (-\alpha)^n q^{2n^2}}{(1-\alpha)(q^2; q^2)_n}
{_2\phi_1} \left({{q^{-n}, \alpha q^n} \atop { -\alpha}} ;  q, q\beta \right).
\end{align*}
\end{prop}
\section{Some limiting cases of Watson's $q$-analog of Whipple's theorem}
Watson's $q$-analog of Whipple's theorem (see, for example, \cite[Eq. (2.5.1)]{Gas+Rah} and \cite[Theorem~5]{Liu})
 can be stated in the following theorem.
\begin{thm}\label{WatsonWhipplethm} {\rm (Watson)} If $n$ is a nonnegative integer, then,  we have
\begin{align*}
&\frac{(\alpha q, \alpha ab/q; q)_n}{(\alpha a, \alpha b; q)_n}
{_4\phi_3}\left({{q^{-n}, q/a, q/b, \alpha cd/q }\atop{\alpha c, \alpha d, q^2/{\alpha ab q^n}}}; q, q\right)\\
&={_8\phi_7}\({{q^{-n}, q\sqrt{\alpha}, -q\sqrt{\alpha}, \alpha, q/a, q/b, q/c, q/d}
\atop{\sqrt{\alpha}, -\sqrt{\alpha}, \alpha a, \alpha b, \alpha c, \alpha d, \alpha q^{n+1}}}; q, \frac{\alpha^2 abcdq^n}{q^2}\).
\end{align*}
\end{thm}
In this section we discuss some limiting cases of Watson's $q$-analog of Whipple's theorem, which will be used
in this paper.
\begin{prop}\label{WWpp1} For any nonnegative integer $n,$ we have
\begin{align*}
&{_3\phi_2}\left({{q^{-n}, \alpha q^{n+1}, \alpha cd/q }\atop{\alpha c, \alpha d}}; q, q\right)\\
&=(-\alpha)^n q^{n(n+1)/2}\frac{(q; q)_n}{(q\alpha; q)_n}
\sum_{j=0}^n (-1)^j\frac{(1-\alpha q^{2j}) (\alpha, q/c, q/d; q)_j}
{(1-\alpha)(q, \alpha c, \alpha d; q)_j}\(\frac{ cd}{q}\)^{j} q^{-j(j+1)/2}.
\end{align*}
\end{prop}
\begin{proof}
Taking $q/a=\alpha q^{n+1}$ in Theorem~\ref{WatsonWhipplethm} and simplifying, we find that
\begin{align}
&\frac{(\alpha q, q^2/b; q)_n }{(q, \alpha b; q)_n}\(\frac{b}{q}\)^n
{_4\phi_3}\left({{q^{-n}, \alpha q^{n+1}, q/b, \alpha cd/q }\atop{\alpha c, \alpha d, q^2/b}}; q, q\right)\label{wweqn1}\\
&=\sum_{j=0}^n \frac{(1-\alpha q^{2j}) (\alpha, q/b, q/c, q/d; q)_j}
{(1-\alpha)(q, \alpha b, \alpha c, \alpha d; q)_j}\(\frac{\alpha bcd}{q^2}\)^{j}.\nonumber
\end{align}
Letting $b\to \infty$ in the above equation and simplifying, we complete the proof of Proposition \ref{WWpp1}.
\end{proof}
Letting $d\to \infty$ in Proposition \ref{WWpp1}, we obtain the following proposition.
\begin{prop}\label{WWpp2} For any nonnegative integer $n,$ we have
\begin{align*}
{_2\phi_1}\left({{q^{-n}, \alpha q^{n+1}}\atop{\alpha c}}; q, c\right)
&=(-\alpha)^n q^{n(n+1)/2}\frac{(q; q)_n}{(q\alpha; q)_n}\\
&\qquad\times
\sum_{j=0}^n \frac{(1-\alpha q^{2j}) (\alpha, q/c; q)_j}
{(1-\alpha)(q, \alpha c; q)_j}c^j \alpha^{-j}  q^{-j^2-j}.
\end{align*}
\end{prop}
Setting $b=0$ in (\ref{wweqn1}), we are led to the following proposition.
\begin{prop}\label{WWpp3} For any nonnegative integer $n,$ we have
\begin{align*}
&(-1)^n\frac{(\alpha q; q)_n }{(q; q)_n}q^{n(n+1)/2}
{_3\phi_2}\left({{q^{-n}, \alpha q^{n+1}, \alpha cd/q }\atop{\alpha c, \alpha d}}; q, 1\right)\label{wweqn1}\\
&=\sum_{j=0}^n (-1)^j\frac{(1-\alpha q^{2j}) (\alpha, q/c, q/d; q)_j}
{(1-\alpha)(q, \alpha c, \alpha d; q)_j}q^{j(j-3)/2}\(\alpha cd\)^{j}.
\end{align*}
\end{prop}
Putting $d=0$ in Proposition~\ref{WWpp3}, we obtain the following proposition.
\begin{prop}\label{WWpp4} For any nonnegative integer $n,$ we have
\begin{align*}
&(-1)^n\frac{(\alpha q; q)_n }{(q; q)_n}q^{n(n+1)/2}
{_2\phi_1}\left({{q^{-n}, \alpha q^{n+1} }\atop{\alpha c}}; q, 1\right)\\
&=\sum_{j=0}^n \frac{(1-\alpha q^{2j}) (\alpha, q/c; q)_j}
{(1-\alpha)(q, \alpha c; q)_j}q^{j^2-j}\(\alpha c\)^{j}.
\end{align*}
\end{prop}
Letting $d \to \infty$ in Proposition~\ref{WWpp3}, we obtain the following Proposition.
\begin{prop}\label{WWpp5} For any nonnegative integer $n,$ we have
\begin{align*}
&(-1)^n\frac{(\alpha q; q)_n }{(q; q)_n}q^{n(n+1)/2}
{_2\phi_1}\left({{q^{-n}, \alpha q^{n+1} }\atop{\alpha c}}; q, \frac{c}{q} \right)\\
&=\sum_{j=0}^n \frac{(1-\alpha q^{2j}) (\alpha, q/c; q)_j}
{(1-\alpha)(q, \alpha c; q)_j}\( \frac{c}{q}\)^{j}.
\end{align*}
\end{prop}
\section{Applications of Theorems~\ref{liunewthmb}, \ref{liunewthmc}, \ref{liunewthmd}, \ref{liunewthme}}
To discuss the applications of Theorems~\ref{liunewthmb}, \ref{liunewthmc}, \ref{liunewthmd}, \ref{liunewthme},
we first introduce some notations.  Ramanujan's theta functions $\phi(q)$ and $\psi(q)$ are defined as follows
\[
\phi(q)=\sum_{n=-\infty}^\infty q^{n^2} \quad \text{and} \quad
\psi(q)=\sum_{n=0}^\infty q^{n(n+1)/2}.
\]
The Gauss identities for  $\phi(q)$ and $\psi(q)$ (see, for example, \cite[p. 347]{Liu2010pacific}) are given by
\begin{equation}
\phi(q)=\prod_{n=1}^\infty \(\frac{1-q^n}{1+q^n}\)\quad \text{and}\quad
\psi(q)=\prod_{n=1}^\infty \(\frac{1-q^{2n}}{1-q^{2n-1}}\).
\label{R:eqn1}
\end{equation}
If $m\ge 1$ is an integer, we use $r_m(n)$ to denote the number of representations of $n$ as a sum of $m$ squares and
we also let $t_m(n)$ denote the number of representations of $n$ as a sum of $m$ triangular numbers.  It is easily seen
that $\phi^m(q)$ is the generating function of $r_m(n)$ and $\psi^m(q)$ is the generating function of $t_m(n).$
As a result, to obtain expressions for $r_m(n)$ and $t_m(n)$, it suffices to obtain expressions for
$\phi^m(q)$ and $\psi^m(q).$ Some important progress in the study of sums of squares have been made recently (see, for
example, \cite{ChanChua}, \cite{Liu2004}, \cite{Milne} and \cite{Ono}). In this section, we shall
use Theorems~\ref{liunewthmb}, \ref{liunewthmc}, \ref{liunewthmd}, \ref{liunewthme} to derive some unusual formulas
involving $\phi^m(q)$ and $\psi^m(q).$   For any integer $r,$
some general formulas for $(q; q)^r_\infty$ are also given.

 We also need the finite theta functions $S_n(q)$ and $T_n(q)$ which are defined as follows
\begin{equation}
S_n(q)=\sum_{j=-n}^n (-1)^j q^{-j^2} \quad \text{and}\quad T_n(q)=\sum_{j=0}^n q^{-j(j+1)/2}.
\label{m:eqn1}
\end{equation}

 Putting $\alpha=q$ and $c_1=c_2=\cdots=c_m=1$ in Theorem~\ref{liunewthmd}, we
 obtain the following proposition.
 \begin{prop}\label{liupp1} If $m$ is a nonnegative integer, then,  we have
 \[
 \prod_{j=1}^\infty (1-q^j)^{m+1}=\sum_{n=0}^\infty (-1)^n q^{n(n-1)/2}(1-q^{2n+1}){_{m+2}\phi_{m+1}}\left( {{q^{-n},  q^{n+1}, q, \cdots, q}
\atop{0, 0, \cdots, 0}}; q, q\right).
\]
 \end{prop}
 Letting $\alpha \to 1$ and $c_1=c_2=\cdots=c_m=q$ in Theorem~\ref{liunewthmd}, we
 obtain the following proposition.
 \begin{prop}\label{liupp2} If $m$ is a nonnegative integer, then,  we have
 \[
 \prod_{j=1}^\infty (1-q^j)^{m+1}=1+\sum_{n=1}^\infty (-1)^n (1+q^n)q^{n(n-1)/2}{_{m+2}\phi_{m+1}}\left( {{q^{-n},  q^{n}, q, \cdots, q}
\atop{0, 0, \cdots, 0}}; q, q\right).
\]
 \end{prop}
 Putting $\alpha=q$ and $b=b_1=b_2=\cdots=b_m=1$ in Theorem~\ref{liunewthme}, we
 obtain the following proposition.
 \begin{prop}\label{liupp3} If $m$ is a nonnegative integer, then,  we have
 \[
 \prod_{j=1}^\infty \frac{1}{(1-q^j)^{m}}=\sum_{n=0}^\infty (-1)^n (1-q^{2n+1})q^{n(n-1)/2}{_{m+2}\phi_{m+1}}\left( {{q^{-n},  q^{n+1}, 0, \cdots, 0}
\atop{q, q, \cdots, q}}; q, q\right).
\]
 \end{prop}
 Letting $\alpha \to 1$ and $b=b_1=b_2=\cdots=b_m=q$ in Theorem~\ref{liunewthme}, we
 obtain the following proposition.
 \begin{prop}\label{newliup4} If $m$ is a nonnegative integer, then,  we have
  \[
 \prod_{j=1}^\infty \frac{1}{(1-q^j)^{m}}=\sum_{n=0}^\infty (-1)^n(1+q^n) q^{n(n-1)/2}{_{m+2}\phi_{m+1}}\left( {{q^{-n},  q^{n}, 0, \cdots, 0}
\atop{q, q, \cdots, q}}; q, q\right).
\]
 \end{prop}
If we put $\alpha=1, a=0,  c_1=c_2=\cdots=c_m=q$ and $b=b_1=b_2=\cdots=b_m=-q$ in Theorem \ref{liunewthmb},
we are led to the following identity:
 \begin{align}
&\prod_{n=1}^\infty \(\frac{1-q^n}{1+q^n}\)^{m+1}
 \label{R:eqn2}\\
 &=1+\sum_{n=1}^\infty (-1)^n (1+q^n) q^{n(n-1)/2} {_{m+2}\phi_{m+1}}\left( {{q^{-n},  q^n, q, \cdots,  q}
\atop{-q, -q, \cdots, -q}}; q, q\right).\nonumber
 \end{align}
 Using the $q$-Chu-Vandermonde summation (see, for example, \cite[Eq. (1.5.3)]{Gas+Rah}), we easily deduce that
 \begin{equation}
 {_2\phi_1}\left( {{q^{-n},  q^n}
\atop{-q}}; q, q\right)=\frac{(-q^{1-n}; q)_n q^{n^2}}{(-q; q)_n}
=\frac{2q^{n(n+1)/2}}{1+q^n}.
\label{R:eqn3}
 \end{equation}
 Taking $m=0$ in (\ref{R:eqn2}) and then using the above in the resulting equation, we
 arrive at the first Gauss identity in (\ref{R:eqn1}). Using this identity, we can rewrite
 (\ref{R:eqn2}) in the following Proposition.
 \begin{prop} \label{newliupp5} If $m$ is a nonnegative integer, then,  we have the identity
  \begin{align*}
&\(\sum_{n=-\infty}^\infty (-1)^n q^{n^2}\)^{m+1}\\
 &=1+\sum_{n=1}^\infty (-1)^n (1+q^n) q^{n(n-1)/2} {_{m+2}\phi_{m+1}}\left( {{q^{-n},  q^n, q, \cdots,  q}
\atop{-q, -q, \cdots, -q}}; q, q\right).
 \end{align*}
 \end{prop}
 If we set $\alpha=1, a=b=b_1=\cdots=b_m=-q$ and $c_1=c_2=\cdots=c_m=q$ in Theorem \ref{liunewthmb} and
 using the first Gauss identity in (\ref{R:eqn1}), we obtain the following proposition.
 \begin{prop} \label{newliupp6} If $m$ is a nonnegative integer, then,  we have the identity
 \begin{align*}
 \(\sum_{n=-\infty}^\infty (-1)^n q^{n^2}\)^{2m+2}=1+2\sum_{n=1}^\infty (-1)^n {_{m+2}\phi_{m+1}}\left( {{q^{-n},  q^n, q, \cdots,  q}
\atop{-q, -q, \cdots, -q}}; q, q\right).
 \end{align*}
 \end{prop}
 Letting $m=0$ in the above proposition and then using (\ref{R:eqn3}), we obtain the well-known identity
 (see, for example, \cite[Eq. (8. 11. 3)]{Gas+Rah})
 \begin{equation}
 \(\sum_{n=-\infty}^\infty (-1)^n q^{n^2}\)^{2}=1+4\sum_{n=1}^\infty (-1)^n\frac{q^{n(n+1)/2}}{1+q^n}.
 \label{R:eqn4}
 \end{equation}
 Using the $q$-Pfaff-Saalsch\"utz summation formula (see, for example, \cite[Eq. (1.7.2)]{Gas+Rah}, \cite{Liu2013}),
 we have
 \begin{equation}
 {_{3}\phi_{2}}\left( {{q^{-n},  q^n, q}
\atop{-q, -q}}; q, q\right)=\frac{(-1; q)^2_n q^n}{(-q; q)^2_n}=\frac{4q^n}{(1+q^n)^2}.
\label{R:eqn5}
 \end{equation}
 Setting $m=1$ in Proposition \ref{newliupp6} and then using (\ref{R:eqn5}), we obtain the well-known identity
 (see, for example, \cite[Eq. (8. 11. 6)]{Gas+Rah})
 \begin{equation}
 \(\sum_{n=-\infty}^\infty (-1)^n q^{n^2}\)^{4}=1+8\sum_{n=1}^\infty (-1)^n\frac{q^n}{(1+q^n)^2}.
 \label{R:eqn6}
 \end{equation}

  \begin{rem}\label{aremark}
 \rm  For $m \ge 2$, there are no known summation formulas for the series $_{m+2}\phi_{m+1}$ in Proposition~ \ref{newliupp6},
 so we can't simplify the formula in Proposition~ \ref{newliupp6} at present.  Finding the summation formulas for these series
 are highly desirable.
  \end{rem}

 Choosing $\alpha=1, a=b=-q, b_1=b_2=\cdots=b_m=0$ and $c_1=c_2=\cdots=c_m=q$ in Theorem \ref{liunewthmb},
  we obtain the following proposition.
  \begin{prop} \label{newliupp7} If $m$ is a nonnegative integer, then,  we have the identity
 \begin{align*}
 \(\sum_{n=-\infty}^\infty (-1)^n q^{n^2}\)^{m+2}&=\prod_{n=1}^\infty \(\frac{1-q^n}{1+q^n}\)^{m+2}
 \\
 &=1+2\sum_{n=1}^\infty (-1)^n {_{m+2}\phi_{m+1}}\left( {{q^{-n},  q^n, q, \cdots,  q}
\atop{-q, 0, \cdots, 0}}; q, q\right).
 \end{align*}
 \end{prop}
This proposition includes Andrews' identity for sums of three squares \cite[Eq. (5.16)]{Andrews86}
(see also \cite[Eq. (7.7)]{Liu}) as a special cases.
\begin{prop} \label{newliupp8} {(\rm Andrews)}  There holds the identity
\begin{equation*}
\(\sum_{n=-\infty}^\infty (-1)^n q^{n^2}\)^{3}
=1+4\sum_{n=1}^\infty (-1)^n \frac{q^n}{1+q^n}
-2\sum_{n=1}^\infty \frac{1-q^n}{1+q^n} \sum_{|j|<n}(-1)^j q^{n^2-j^2}.
\end{equation*}
\end{prop}
\begin{proof}
Taking $m=1$ in proposition \ref{newliupp7}, we immediately conclude that
\begin{equation}
\(\sum_{n=-\infty}^\infty (-1)^n q^{n^2}\)^{3}=
1+2\sum_{n=1}^\infty (-1)^n {_{3}\phi_{2}}\left( {{q^{-n},  q^n, q}
\atop{-q, 0}}; q, q\right).
\label{newR:eqn7}
\end{equation}
If we set $a=1$ and $c=-1$ in \cite[Lemma~4.1]{Liu2013}, we easily deduce that
\begin{equation}
{_{3}\phi_{2}}\left( {{q^{-n},  q^n, q}
\atop{-q, 0}}; q, q\right)=q^{n^2} \frac{(q; q)_n}{(-q; q)_n}
\sum_{j=0}^n \frac{(-1; q)_j  q^{j(1-n)}}{(q; q)_j}.
\label{R:eqn7}
\end{equation}
Using \cite[Eq. (6.1)]{Liu2013}, we easily find that the
inner summation of the right-hand side of the above equation equals
\begin{align*}
&\frac{(-1; q)_n q^{n(1-n)}}{(q; q)_n}
+\sum_{j=0}^{n-1}\frac{(-1; q)_j  q^{j(1-n)}}{(q; q)_j}\\
&=\frac{(-1; q)_n q^{n(1-n)}}{(q; q)_n}+(-1)^{n-1} \frac{(-q; q)_{n-1}}{(q; q)_{n-1}} \sum_{|j|<n}(-1)^j q^{-j^2}.
\end{align*}
Substituting the above equation into (\ref{R:eqn7}) and simplifying , we find that
\begin{align}
&{_{3}\phi_{2}}\left( {{q^{-n},  q^n, q}
\atop{-q, 0}}; q, q\right)\label{R:eqn8}\\
&=\frac{2q^n}{1+q^n}
+(-1)^{n-1} \frac{1-q^n}{1+q^n} \sum_{|j|<n}(-1)^j q^{n^2-j^2}.\nonumber
\end{align}
Combining (\ref{newR:eqn7}) and (\ref{R:eqn8}) we complete the proof of the proposition.
\end{proof}
Replacing $q$ by $q^2$ in Theorem \ref{liunewthmb} and then setting $\alpha=1, a=0,  b=b_1=b_2=\cdots=b_m=q, c_1=c_2=\cdots=c_m=q^2$
in the resulting equation, we find the following identity.
 \begin{align*}
 &\prod_{n=1}^\infty \(\frac{1-q^{2n}}{1-q^{2n-1}}\)^{m+1}\\
 &=1+\sum_{n=1}^\infty (-1)^n (1+q^{2n}) q^{n^2-n} {_{m+2}\phi_{m+1}}\left( {{q^{-2n},  q^{2n}, q^2, \cdots,  q^2}
\atop{q, q, \cdots, q}}; q^2, q^2\right).\nonumber
 \end{align*}
Setting $m=0$ in the above equation and then using the $q$-Chu-Vandermonde summation, we find the
second Gauss identity in (\ref{R:eqn1}). Combining the above equation and  the
second Gauss identity in (\ref{R:eqn1}), we obtain the following proposition.
\begin{prop}\label{newliupp9} There holds the identity
\begin{align*}
 & \(\sum_{n=0}^\infty q^{n(n+1)/2}\)^{m+1}\\
 &=1+\sum_{n=1}^\infty (-1)^n (1+q^{2n}) q^{n^2-n} {_{m+2}\phi_{m+1}}\left( {{q^{-2n},  q^{2n}, q^2, \cdots,  q^2}
\atop{q, q, \cdots, q}}; q^2, q^2\right).
 \end{align*}
\end{prop}
\begin{prop}\label{newliupp10}If $m$ is a nonnegative integer, then,  we have the formula
 \begin{align*}
&\(\sum_{n=0}^\infty q^{n(n+1)/2}\)^{m+2}\\
 &=\sum_{n=0}^\infty \frac{(1+q^{2n+1})q^{-n}}{1-q} {_{m+2}\phi_{m+1}}\left( {{q^{-2n},  q^{2n+2}, q^2, \cdots,  q^2}
\atop{q^3, 0, \cdots, 0}}; q^2, q^2\right).
 \end{align*}
\end{prop}
\begin{proof}
Letting $b_1=b_2=\cdots=b_m=0$ and $c_1=c_2=\cdots=c_m=q/\alpha$ in Theorem \ref{liunewthmb}, we
deduce that
 \begin{align*}
&  \frac{(\alpha, \alpha ab/q; q)_\infty (q; q)_\infty^m}
{(\alpha a, \alpha b; q)_\infty (a; q)_\infty^m} \\
&=\sum_{n=0}^\infty  \frac{(1-\alpha q^{2n}) (\alpha, q/a; q)_n (a/q)^n}
{(q, \alpha a; q)_n}
{_{m+2}\phi_{m+1}}\left( {{q^{-n}, \alpha q^n, q, \cdots, q}
\atop{\alpha b, 0, \cdots, 0}}; q, q\right).
 \end{align*}
Replacing $q$ by $q^2$ in the above equation,  setting $a=b=q, \alpha=q^2$ in the
resulting equation, using the second Gauss identity in (\ref{R:eqn1}),
we complete the proof of Proposition \ref {newliupp10}.
\end{proof}
Setting $m=0$ in Proposition \ref {newliupp10} and then using the $q$-Chu-Vandermonde summation, we find
that
\begin{equation}
\(\sum_{n=0}^\infty q^{n(n+1)/2}\)^2=\sum_{n=0}^\infty \frac{(-1)^n(1+q^{2n+1})q^{n^2+n}}{1-q^{2n+1}}.
\label{R:eqn9}
\end{equation}
Taking $m=1$ in Proposition \ref {newliupp10} and then using \cite[Lemma~4.1]{Liu2013}, we can find
Andrews' identity for sums of three triangular numbers \cite[Eq. (5.17)]{Andrews86}, \cite[Theorem~8]{Liu}:
\begin{equation}
\(\sum_{n=0}^\infty q^{n(n+1)/2}\)^3
=\sum_{n=0}^\infty \sum_{j=0}^{2n} \frac{(1+q^{2n+1})q^{2n^2+2n-j(j+1)/2}}{1-q^{2n+1}}.
\label{R:eqn10}
\end{equation}

We end this section by proving the following theorem using Theorems~\ref{liunewthmb} and the Sears $_4\phi_3$ transformation.
\begin{thm}\label{Liutripthm} If $\max\{|\alpha a|, |\alpha b|, |\alpha ac/q|\}<1,$ then,  we have
\begin{align*}
&\frac{(q\alpha, \alpha c, \alpha ab/q; q)_\infty}{(\alpha a, \alpha b, \alpha ac/q; q)_\infty}\\
&=\sum_{n=0}^\infty  \frac{(1-\alpha q^{2n})(\alpha, q/a, q/b; q)_n (-\alpha ab)^nq^{n(n-3)/2}}{(1-\alpha)(q, \alpha a, \alpha b; q)_n}
{_2\phi_1}\({{q^{-n}, \alpha q^n}\atop{q/b}}; q, \frac{qc}{b}\).
\end{align*}
\end{thm}
\begin{proof}
We start with the case $m=1$ of  Theorems~\ref{liunewthmb}, which states
\begin{align}
&  \frac{(\alpha q, \alpha ab/q, \alpha a b_1/q, \alpha c_1 ; q)_\infty}
{(\alpha a, \alpha b, \alpha a c_1/q, \alpha b_1 ; q)_\infty} \label{R:eqn11}\\
&=\sum_{n=0}^\infty  \frac{(1-\alpha q^{2n}) (\alpha, q/a; q)_n (a/q)^n}
{(1-\alpha)(q, \alpha a; q)_n}
{_{3}\phi_{2}}\left( {{q^{-n}, \alpha q^n, \alpha c_1}
\atop{\alpha b, \alpha b_1}}; q, q\right).\nonumber
 \end{align}
The Sears $_4\phi_3$ transformation  (see, for example,  \cite[p. 71]{Gas+Rah}) can be restated as follows
\begin{align*}
&{_4\phi_3} \left({{q^{-n}, \alpha q^n, \beta, \gamma} \atop {c, d, q\alpha \beta \gamma/cd}} ;  q, q \right)\\
&=\frac{(q\alpha/c, cd/\beta \gamma; q)_n}{(c, q\alpha \beta \gamma; q)_n} \left(\frac{\beta \gamma}{d}\right)^n
{_4\phi_3} \left({{q^{-n}, \alpha q^n, d/\beta, d/\gamma} \atop {d, dc/\beta \gamma,  q\alpha/c}} ;  q, q \right).
\end{align*}
Setting $\gamma=0$ in the above equation,  we immediately  deduce that
\begin{align}
&{_3\phi_2} \left({{q^{-n}, \alpha q^n, \beta} \atop {c, d}} ;  q, q \right)
\label{R:eqn12}\\
&=(-c)^n q^{n(n-1)/2} \frac{(q\alpha/c; q)_n}{(c; q)_n}
{_3\phi_2} \left({{q^{-n}, \alpha q^n, d/\beta} \atop {d,  q\alpha/c}} ;  q, \frac{q\beta}{c} \right).
\nonumber
\end{align}
Applying this transformation formula to the $_3\phi_2$ series on the right-hand side of (\ref{R:eqn11}),
we conclude that
\begin{align*}
&  \frac{(\alpha q, \alpha ab/q, \alpha a b_1/q, \alpha c_1 ; q)_\infty}
{(\alpha a, \alpha b, \alpha a c_1/q, \alpha b_1 ; q)_\infty} \\
&=\sum_{n=0}^\infty  \frac{(1-\alpha q^{2n}) (\alpha, q/a, q/b; q)_n (-\alpha ab)^n q^{n(n-3)/2}}
{(1-\alpha)(q, \alpha a, \alpha b; q)_n}
{_3\phi_2}\left( {{q^{-n}, \alpha q^n, b_1/c_1}
\atop{\alpha b_1, q/b}}; q, \frac{qc_1}{b}\right).
 \end{align*}
 Putting $b_1=0$ in the above equation and then  replacing $c_1$ by $c,$ we complete the proof of
 Theorem~\ref{Liutripthm}.
\end{proof}
\begin{rem}
\rm By taking $c_1=q/\alpha, b_1=0$ in (\ref{R:eqn11}) and then using \cite[Lemma~4.1]{Liu2013}, we can obtain
the following identity of Andrews \cite[Theorem~5]{Andrews86} (see also \cite[Theorem~1.4]{Liu2013}).
\begin{thm}\label{andrewsthm} For $\max \{|a|, |\alpha a|, |\alpha b|\}<1,$ we have
\begin{align*}
&\frac{(q, \alpha q, \alpha ab/q; q)_\infty}{(\alpha a, \alpha b, a; q)_\infty}\\
&=\sum_{n=0}^\infty \frac{(1-\alpha q^{2n})(\alpha, q/a; q)_n (\alpha a)^n q^{n^2-n}}
{(1-\alpha)(\alpha a, \alpha b; q)_n} \sum_{j=0}^n \frac{(\alpha b/q; q)_j \alpha^{-j}q^{j(1-n)}}{(q; q)_j}.
\end{align*}
\end{thm}
When $b=c,$  Theorem~\ref{Liutripthm} reduces to the well-known identity
\[
\sum_{n=0}^\infty \frac{(1-\alpha q^{2n})(\alpha, q/a; q)_n (\alpha a)^n q^{n(n-1)}}
{(q, \alpha a; q)_n}=\frac{(\alpha; q)_\infty}{(\alpha a; q)_\infty}.
\]
\end{rem}
Using  Theorem~\ref{Liutripthm}, we can prove the following proposition.
\begin{prop}\label{Liutrippa} For $|a|<1,$ we have
\[
\frac{(q; q)_\infty^2 (-a; q)_\infty}{(a; q)^2_\infty(-q; q)_\infty}
=\sum_{n=0}^\infty \sum_{j=-n}^n (-1)^{n+j} (1-q^{2n+1}) q^{n^2-j^2} \frac{(q/a; q)_n a^n}{(a; q)_{n+1}}.
\]
\end{prop}
\begin{proof} Setting $\alpha=q, c=1$ and $b=-1$ in Theorem~\ref{Liutripthm}, we deduce that
\[
\frac{(q; q)_\infty^2 (-a; q)_\infty}{(a; q)^2_\infty(-q; q)_\infty}
=\sum_{n=0}^\infty (1-q^{2n+1}) \frac{(q/a; q)_n a^nq^{n(n-1)/2}}{(a; q)_{n+1}}
{_2\phi_1}\left( {{q^{-n},  q^{n+1}}
\atop{-q}}; q, -q\right).
\]
Setting $\alpha=1, c=-q$ in Proposition~\ref{WWpp2} and simplifying, we have
\[
{_2\phi_1}\left( {{q^{-n},  q^{n+1}}
\atop{-q}}; q, -q\right)=(-1)^n q^{n(n+1)/2}\sum_{j=-n}^n (-1)^j q^{-j^2}.
\]
Combining the above two equations, we finish the proof of Proposition~\ref{Liutrippa}.
\end{proof}
Putting $a=0$ in Proposition~\ref{Liutrippa}, we obtain the Andrews identity \cite[Eq. (5.15)]{Andrews86},
\cite[Eq. (7.8)]{Liu}
\begin{equation}
(q; q)^2_\infty (q; q^2)_\infty=\sum_{n=0}^\infty \sum_{j=-n}^n (-1)^{j} (1-q^{2n+1}) q^{{(3n^2+n)/2}-j^2}.
\label{heckeeqn1}
\end{equation}
Putting $a=-q^{1/2}$ in Proposition~\ref{Liutrippa} and then replacing $q$ by $q^2$, we deduce that
\[
\frac{(q; q)_\infty (q^2; q^2)_\infty}
{(-q; q)_\infty(-q; q^2)_\infty}=\phi(-q)\psi(-q)=\sum_{n=0}^\infty \sum_{j=-n}^n (-1)^{j}(1-q^{2n+1}) q^{2n^2+n-2j^2}.
\]
Replacing $q$ by $-q$ in the above equation, we are led to the following beautiful identity:
\begin{equation}
\phi(q)\psi(q)=\sum_{n=0}^\infty \sum_{j=-n}^n (-1)^{n+j}(1+q^{2n+1}) q^{2n^2+n-2j^2}.
\label{heckeeqn2}
\end{equation}

Using  Theorem~\ref{Liutripthm}, we can also prove the following proposition.
\begin{prop}\label{Liutrippb} For $|a|<1,$ we have
\[
\frac{(q^2; q^2)_\infty^2 (aq; q^2)_\infty}{(a; q^2)^2_\infty(q; q^2)_\infty}
=\sum_{n=0}^\infty \sum_{j=-n}^n (1+q^{2n+1}) q^{2n^2+n-2j^2-j} \frac{(q^2/a; q^2)_n a^n}{(a; q^2)_{n+1}}.
\]
\end{prop}
\begin{proof} If we replace $q$ by $q^2$ and then setting $\alpha=q^2, b=q, c=1$, then Theorem~\ref{Liutripthm}
becomes
\[
\frac{(q^2; q^2)_\infty^2 (aq; q^2)_\infty}{(a; q^2)^2_\infty(q; q^2)_\infty}
=\sum_{n=0}^\infty  (1+q^{2n+1})\frac{(q^2/a; q^2)_n (-a)^n q^{n^2}}{(a; q^2)_{n+1}}
{_2\phi_1}\left( {{q^{-2n},  q^{2n+2}}
\atop{q}}; q^2, q\right).
\]
Replacing $q$ by $q^2$ in Proposition~\ref{WWpp2} and then setting $\alpha=1, c=q$, we obtain
\[
{_2\phi_1}\left( {{q^{-2n},  q^{2n+2}}
\atop{q}}; q^2, q\right)=(-1)^n q^{n(n+1)}\sum_{j=-n}^n  q^{-2j^2-j}.
\]
Combining the above two equations, we complete the proof of Proposition~\ref{Liutrippb}.
\end{proof}
Setting $a=q$ in Proposition~\ref{Liutrippb}, we immediately conclude that
\begin{equation}
\(\sum_{j=0}^\infty q^{j(j+1)/2}\)^3=\psi^3(q)=\sum_{n=0}^\infty \sum_{j=-n}^n \(\frac{1+q^{2n+1}}{1-q^{2n+1}}\) q^{2n^2+2n-2j^2-j}.
\label{heckeeqn3}
\end{equation}
This identity is similar to Andrews' identity for sums of three triangular numbers in (\ref{R:eqn10}), which
also implies Gauss's famous result that every integer is the sum of three triangular numbers.

Setting $a=-q$ in Proposition~\ref{Liutrippb} and then replacing $q$ by $-q$, we deduce that
\begin{equation}
\psi(q^2)\psi(q)=\sum_{n=0}^\infty \sum_{j=-n}^n (-1)^j q^{2n^2+2n-2j^2-j}.
\label{heckeeqn4}
\end{equation}
Putting $a=0$ in Proposition~\ref{Liutrippb}, we obtain the following identity, which is
similar to the identity in \cite[Eq. (7.17)]{Liu}:
\begin{equation}
\frac{(q^2; q^2)^2_\infty}{(q; q^2)_\infty}=\sum_{n=0}^\infty \sum_{j=-n}^n
(-1)^n(1+q^{2n+1}) q^{3n^2+2n-2j^2-j}.
\label{heckeeqn5}
\end{equation}
\section {Hecke-type series identities}
We \cite{Liu2013} have proved the following general $q$-transformation formula \cite[Theorem~1.6]{Liu2013} using Theorem \ref{liuthm2}.
\begin{thm}\label{mliuthma} If  $\{A_n\}$ is a complex sequence, then, under suitable convergence conditions,
 we have
\begin{align*}
&\frac{(\alpha q, \alpha ab/q; q)_\infty}
{(\alpha a, \alpha b; q)_\infty} \sum_{n=0}^\infty
A_n (q/a; q)_n (\alpha a)^n \\
&=\sum_{n=0}^\infty  \frac{(1-\alpha q^{2n}) (\alpha, q/a, q/b; q)_n (-\alpha ab/q)^n q^{n(n-1)/2}}
{(1-\alpha)(q, \alpha a, \alpha b; q)_n}
\sum_{k=0}^n \frac{(q^{-n}, \alpha q^n; q)_k (q^2/b)^k}{(q/b; q)_k} A_k.
\end{align*}
\end{thm}
\begin{rem} \rm The condition ``independent of $a$" in \cite[Theorem~1.6]{Liu2013} is not necessary.
\end{rem}
The main result of this section is the following theorem, which can be
derived from Theorem \ref{mliuthma} by choosing
\[
A_k=\frac{(q/b, \beta, \gamma; q)_k (bz/q)^k}{(q, c, d, h; q)_k}.
\]
\begin{thm}\label{mliuthmb} For $|\alpha abz/q|<1$, we have the $q$-transformation formula
\begin{align*}
&\frac{(\alpha q, \alpha ab/q; q)_\infty}
{(\alpha a, \alpha b; q)_\infty} {_4\phi_3} \left({{q/a, q/b, \beta, \gamma} \atop { c, d, h}} ;  q, \frac{\alpha ab z}{q} \right) \\
&=\sum_{n=0}^\infty \frac{(1-\alpha q^{2n}) (\alpha, q/a, q/b; q)_n (-\alpha ab/q)^n q^{n(n-1)/2}}
{(1-\alpha)(q, \alpha a, \alpha b; q)_n}
{_4\phi_3} \left({{q^{-n}, \alpha q^n, \beta, \gamma} \atop {c, d, h}} ;  q, qz \right).
\end{align*}
\end{thm}
Now we will begin to derive Hecke-type series identities using Theorem \ref{mliuthmb}.

\subsection{The proof of Theorem \ref{liunewthmf}}
Setting $\alpha=q, c=-q, z=-1, d=h=\beta=\gamma=0$ in Theorem \ref{mliuthmb}, we deduce that
\begin{align*}
&\frac{(q, ab; q)_\infty}{(qa, qb; q)_\infty}\sum_{n=0}^\infty \frac{(q/a, q/b; q)_n (-ab)^n}{(q^2; q^2)_n}\\
&=\sum_{n=0}^\infty \frac{(1-q^{2n+1})(q/a, q/b; q)_n (-ab)^n q^{n(n-1)/2}}{(qa, qb; q)_n}
{_2\phi_1} \left({{q^{-n}, q^{n+1}} \atop { -q}} ;  q, -q \right).
\end{align*}
By setting $\alpha=1$ and $c=-q$ in Proposition~\ref{WWpp2}, we can easily find that
\begin{equation}
{_2\phi_1} \left({{q^{-n}, q^{n+1}} \atop { -q}} ;  q, -q \right)
=(-1)^n q^{n(n+1)/2} \sum_{j=-n}^n (-1)^j q^{-j^2},
\label{ram:eqn1}
\end{equation}
Combining the above two equations, we finish the proof of Theorem~\ref{liunewthmf}.
\subsection{} Setting $d=h=\gamma=0, \beta=q$ and $z=1$, replacing $c$ by $qc$ in Theorem \ref{mliuthmb} and
then using \cite[Lemma~4.1]{Liu2013} we can prove the following $q$-formula \cite[Theorem~1.9]{Liu2013},
which has many applications to Hecke-type series identities (see \cite{Liu2013} for the details).
\begin{thm}\label{mliuthmc} For $|\alpha ab/q|<1$, we have the $q$-identity
\begin{align*}
&\frac{(q\alpha, \alpha ab/q; q)_\infty}{(\alpha a, \alpha b; q)_\infty}\sum_{n=0}^\infty \frac{(q/a, q/b; q)_n (\alpha ab/q)^n}{(cq; q)_n}\\
&=\sum_{n=0}^\infty \frac{(1-\alpha q^{2n}) (\alpha, q/a, q/b; q)_n (-\alpha^2 ab)^n q^{3n(n-1)/2}}
{(1-\alpha)(qc, \alpha a, \alpha b; q)_n}
\sum_{j=0}^n \frac{(c; q)_j\alpha^{-j} q^{j(1-n)}}{(q; q)_j}.
\end{align*}
\end{thm}
\subsection{} In this subsection, we will prove the following theorem using Theorem~\ref{mliuthmb}.
\begin{thm}\label{mliuthmd} For $|ab/q|<1, $ we have the $q$-transformation formula
\begin{align*}
&\frac{(q^2,  ab; q^2)_\infty}
{(q^2a,  q^2b; q^2)_\infty} \sum_{n=0}^\infty \frac{(q^2/a, q^2/b; q^2)_n (ab/q)^n}
{(q; q)_{2n}}\\
&=\sum_{n=0}^\infty \sum_{j=-n}^n (1-q^{4n+2})q^{2n^2-2j^2-j}\frac{ (q^2/a, q^2/b; q^2)_n (ab)^n }
{(q^2a, q^2 b; q^2)_n}.
\end{align*}
\end{thm}
\begin{proof} Replacing $q$ by $q^2$ in Theorem \ref{mliuthmb} and then
setting $d=h=\beta=\gamma=0, \alpha=q^2, c=q$ and $z=q^{-1},$ we obtain
\begin{align*}
&\frac{(q^2,  ab; q^2)_\infty}
{(q^2a,  q^2b; q^2)_\infty} {_2\phi_1} \left({{q^2/a, q^2/b} \atop { q}} ;  q^2, \frac{ab }{q} \right) \\
&=\sum_{n=0}^\infty (1-q^{4n+2})\frac{ (q^2/a, q^2/b; q^2)_n (-ab)^n q^{n(n-1)}}
{(q^2a, q^2 b; q^2)_n}
{_2\phi_1} \left({{q^{-2n}, q^{2n+2}} \atop {q}} ;  q^2, q \right).
\end{align*}
Replacing $q$ by $q^2$ in Proposition \ref{WWpp2} and then setting $\alpha=1$ and $c=q$, we obtain
\[
{_2\phi_1}\left({{q^{-2n},  q^{2n+2}}\atop{q}}; q^2, q\right)
=(-1)^n q^{n(n+1)}\sum_{j=-n}^n q^{-2j^2-j}.
\]
Combining the above two equations, we complete the proof of Theorem~\ref{mliuthmd}.
\end{proof}
Setting $(a, b)=(0, 0), (1, 0), (q, -q)$ in Theorem~\ref{mliuthmd} respectively,  we obtain
the following three Hecke-type series identities.
\begin{align}
&\sum_{n=0}^\infty \frac{q^{2n^2+n}}{(q; q)_{2n}}
=\frac{1}{(q^2; q^2)_\infty}
\sum_{n=0}^\infty \sum_{j=-n}^n(1-q^{4n+2}) q^{4n^2-2j^2+2n-j} ,
\label{ram:eqn2}\\
&\sum_{n=0}^\infty \frac{(-1)^n q^{n^2}}{(q; q^2)_n}
=\sum_{n=0}^\infty \sum_{j=-n}^n (-1)^n (1-q^{4n+2}) q^{3n^2-2j^2+n-j},
\label{ram:eqn3}\\
&\sum_{n=0}^\infty \frac{(q^2; q^4)_n (-q)^n}{ (q; q)_{2n}}
=\frac{(q^2; q^4)_\infty}{(q^4; q^4)_\infty}
\sum_{n=0}^\infty \sum_{j=-n}^n (-1)^n q^{2n^2+2n-2j^2-j}.
\label{ram:eqn4}
\end{align}
The identity in (\ref{ram:eqn2}) is equivalent to the identity in \cite[Corollary~5.4]{Rowell}.
\subsection{} In this subsection,  we will set up the following transformation formula.
\begin{thm}\label{mliuthme} For $|ab|<1,$ we have the $q$-transformation formula
\begin{align*}
&\frac{(q^2, ab; q^2)_\infty}{(q^2a, q^2b; q^2)_\infty}
\sum_{n=0}^\infty \frac{(q^2/a, q^2/b, q; q^2)_n (ab)^n}{(q^2; q^2)_n (-q; q)_{2n}}\\
&=\sum_{n=0}^\infty \sum_{j=-n}^n (-1)^j (1-q^{4n+2}) q^{2n^2-j^2}
\frac{(q^2/a, q^2/b, q; q^2)_n (ab)^n}{(q^2a; q^2b; q^2)_n }.
\end{align*}
\end{thm}
\begin{proof} It is easily seen that using Proposition \ref{WWpp1}, one can
prove that
\begin{equation}
{_3\phi_2}\left({{q^{-2n},  q^{2n+2}, q}\atop{{-q}, -{q^2}}}; q^2, q^2\right)
=(-1)^n q^{n^2+n}S_n(q),
\label{ram:eqn5}
\end{equation}
which can be also found in \cite[p. 30, Eq. (6.15)]{Andrews2012}. Replacing $q$ by $q^2$ in
Theorem~\ref{mliuthmb} and then putting $ \alpha=q^2, \beta=q, c=-q, d=-q^2, \gamma=h=0$ and $z=1,$ we deduce
that
\begin{align*}
&\frac{(q^2, ab; q^2)_\infty}{(q^2a, q^2b; q^2)_\infty}
\sum_{n=0}^\infty \frac{(q^2/a, q^2/b, q; q^2)_n (ab)^n}{(q^2; q^2)_n (-q; q)_{2n}}\\
&=\sum_{n=0}^\infty  (1-q^{4n+2})
\frac{(q^2/a, q^2/b, q; q^2)_n (-ab)^nq^{n^2-n}}{(q^2a; q^2b; q^2)_n }
{_3\phi_2}\left({{q^{-2n},  q^{2n+2}, q}\atop{{-q}, -{q^2}}}; q^2, q^2\right).
\end{align*}
Substituting (\ref{ram:eqn5}) into the right-hand side of the above equation, we
complete the proof of Theorem \ref{mliuthme}.
\end{proof}
Setting $a=b=0$ in Theorem \ref{mliuthme}, we obtain the following identity of Andrews
\cite[Eq.(1.16)]{Andrews2012}
\begin{prop}\label{mpp1} {\rm (Andrews)} We have the Hecke-type series identity
\[
\sum_{n=0}^\infty \frac{(q; q^2)_n q^{2n^2+2n}}{(q^2; q^2)_n(-q; q)_{2n}}
=\frac{1}{(q^2; q^2)_\infty}\sum_{n=0}^\infty \sum_{j=-n}^n (-1)^j (1-q^{4n+2}) q^{4n^2+2n-j^2}.
\]
\end{prop}
Putting $a=1$ and $b=0$ in Theorem \ref{mliuthme}, we obtain the identity
\begin{equation}
\sum_{n=0}^\infty \frac{(-1)^n (q; q^2)_n q^{n^2+n}}{(-q; q)_{2n}}
=\sum_{n=0}^\infty \sum_{j=-n}^n (-1)^{j+n} (1-q^{4n+2}) q^{3n^2+n-j^2}.
\label{ram:eqn6}
\end{equation}
Taking $a=b=q$ in Theorem \ref{mliuthme} and simplifying, we find that
\begin{equation}
\sum_{n=0}^\infty \frac{(q; q^2)^3 q^{2n}}{(q^2; q^2)_n(-q; q)_{2n}}
=\frac{(q; q^2)_\infty^2}{(q^2; q^2)_\infty^2}
\sum_{n=0}^\infty \sum_{j=-n}^n (-1)^j \frac{1+q^{2n+1}}{1-q^{2n+1}} q^{2n^2+n-j}.
\label{ram:eqn7}
\end{equation}
\subsection{} The following $q$-formula is established in this subsection by using Theorem~\ref{mliuthmb}.
\begin{thm}\label{mliuthmf} If $T_n$ is defined as in (\ref{m:eqn1}) and $|ab/q|<1,$ then,  we have
\begin{align*}
&\frac{(q, ab/q; q)_\infty}{(a, b; q)_\infty} \sum_{n=0}^\infty \frac{(-q; q)_n^2(q/a, q/b; q)_n (ab/q)^n}{(q; q)_{2n}}\\
&=1+\sum_{n=1}^\infty (1+q^n)\frac{(q/a, q/b; q)_n(ab)^n}{(a, b;q)_n} q^{n^2-2n}(q^nT_n(q)-T_{n-1}(q)).
\end{align*}
\end{thm}
\begin{proof} Setting $\beta=-q, z=1, h=\gamma, c=q^{1/2}, d=-q^{1/2}$ and letting $\alpha \to 1$ in
Theorem~\ref{mliuthmb}, we find that
\begin{align*}
&\frac{(q, ab/q; q)_\infty}{(a, b; q)_\infty} \sum_{n=0}^\infty \frac{(-q; q)_n^2(q/a, q/b; q)_n (ab/q)^n}{(q; q)_{2n}}\\
&=1+\sum_{n=1}^\infty (1+q^n)\frac{(q/a, q/b; q)_n(-ab)^n}{(a, b;q)_n}q^{n(n-3)/2}
{_3\phi_2} \left({{q^{-n}, q^n, -q} \atop {q^{1/2}, -q^{1/2}}} ;  q, q \right).
\end{align*}
Letting $\alpha \to q^{-1}$ and $ c=q^{3/2}, d=-q^{3/2}$ in Proposition~\ref{WWpp1} and simplifying, we can obtain
\[
{_3\phi_2} \left({{q^{-n}, q^n, -q} \atop {q^{1/2}, -q^{1/2}}} ;  q, q \right)
=(-1)^n q^{n(n-1)/2}\(q^n T_n(q)-T_{n-1}(q)\),
\]
which can also be found in \cite[Eq. (5.3)]{Andrews2012}. Combining the above two equations, we complete the proof
of Theorem~\ref{mliuthmf}.
\end{proof}
Setting $a=b=0$ in Theorem~\ref{mliuthmf} and simplifying , we obtain the following identity of Andrews \cite[Eq. (1.11)]{Andrews2012}:
\begin{equation}
\sum_{n=0}^\infty \frac{q^{n^2}(-q; q)^2_n}{(q; q)_{2n}}
=\frac{1}{(q; q)_\infty} \sum_{n=0}^\infty (1-q^{6n+6}) q^{2n^2+n} \sum_{j=0}^n q^{-j(j+1)/2}.
\label{ram:eqn8}
\end{equation}
Setting $b=0$, multiplying both sides by $1-a$ and letting $a=1,$ we arrive at
\begin{align}
&\sum_{n=0}^\infty (-1)^n \frac{(-q; q)_n q^{n(n-1)/2}}{(q; q^2)_n}
\label{ram:eqn9}\\
&=\sum_{n=0}^\infty \sum_{j=0}^n (-1)^n (1-q^{6n+2}) q^{(3n^2-n)/2-j(j+1)/2} .
\nonumber
\end{align}
Using the same argument as that we used in the proof of Theorem~\ref{mliuthmf}, we can prove the
following identity by using \cite[Eq.(5.3)]{Andrews2012}.
\begin{thm}\label{mliuthmg} If $|ab/q|<1$ and $T_n$ is defined as in (\ref{m:eqn1}), then,  we have
\begin{align*}
&\frac{(q, ab/q; q)_\infty}{(a, b; q)_\infty} \sum_{n=0}^\infty \frac{(q/a, q/b; q)_n (ab/q)^n}{(q; q^2)_{n}}\\
&=1+\sum_{n=1}^\infty (1+q^n)\frac{(q/a, q/b; q)_n(-ab)^n}{(a, b;q)_n} q^{n^2-2n}(q^n T_n(q)-T_{n-1}(q)).
\end{align*}
\end{thm}
Setting $a=b=0,$ we are led to the Andrews identity \cite[Eq. (1.10)]{Andrews2012}
\begin{equation}
\sum_{n=0}^\infty \frac{q^{n^2}}{(q; q^2)_n}=\frac{1}{(q; q)_\infty}
\sum_{n=0}^\infty \sum_{j=0}^n (-1)^n (1-q^{6n+6}) q^{2n^2+n-j(j+1)/2}.
\label{ram:eqn10}
\end{equation}
Taking $a=1$ and $b=0$ in Theorem~\ref{mliuthmg},  we can obtain the identity
\begin{equation}
\sum_{n=0}^\infty \frac{(-1)^n (q; q)_n q^{n(n-1)/2}}{(q; q^2)_n}
=\sum_{n=0}^\infty \sum_{j=0}^n (1+q^{6n+2}) q^{(3n^2-n)/2-j(j+1)/2}.
\label{ram:eqn11}
\end{equation}
\subsection{} In this subsection, we will prove the following $q$-formula.
\begin{thm} \label{mliuthmh} For $|ab/q|<1, $ we have
\begin{align*}
&\frac{(q, ab; q)_\infty}{(qa, qb; q)_\infty}\sum_{n=0}^\infty \frac{(q/a, q/b; q)_n (ab/q)^n}{(q^2; q^2)_n}\\
&=\sum_{n=0}^\infty \sum_{j=-n}^n (-1)^{j} (1-q^{2n+1}) q^{j^2}
\frac{(q/a, q/b; q)_n (ab/q)^n}{(qa, qb; q)_n}.
\end{align*}
\end{thm}
\begin{proof} Setting $c=d=\beta=\gamma=0$ and $\alpha=q, c=-q, z=q^{-1}, h=-q$ in Theorem~\ref{mliuthmb}, we
find that
\begin{align*}
&\frac{(q, ab; q)_\infty}{(qa, qb; q)_\infty}\sum_{n=0}^\infty \frac{(q/a, q/b; q)_n (ab/q)^n}{(q^2; q^2)_n}\\
&=\sum_{n=0}^\infty (1-q^{2n+1})\frac{(q/a, q/b; q)_n (-ab)^n q^{n(n-1)/2}}{(qa, qb; q)_n}
{_2\phi_1}\({{q^{-n}, q^{n+1}}\atop{-q}}; q, 1\).
\end{align*}
Taking $\alpha=1$ and $c=-q$ in Proposition~\ref{WWpp4} and simplifying, we find that
\[
{_2\phi_1}\({{q^{-n}, q^{n+1}}\atop{-q}}; q, 1\)=(-1)^nq^{-n(n+1)/2}\sum_{j=-n}^n (-1)^j q^{j^2}.
\]
Combining the above two equations, we complete the proof of the theorem.
\end{proof}
Setting $(a, b)=(0, 0), (1,0)$ and $(-1, 0)$, respectively, in Theorem~\ref{mliuthmh}, we obtain
\begin{equation}
\sum_{n=0}^\infty \frac{q^{n^2}}{(q^2; q^2)_n}
=\frac{1}{(q; q)_\infty}\sum_{n=0}^\infty \sum_{j=-n}^n (-1)^{j} (1-q^{2n+1}) q^{n^2+j^2},
\label{ram;eqn12}
\end{equation}
\begin{equation}
\sum_{n=0}^\infty \frac{(-1)^n q^{n(n-1)/2}}{(-q; q)_n}
=\sum_{n=0}^\infty \sum_{j=-n}^n (-1)^{n+j} (1-q^{2n+1}) q^{j^2+n(n-1)/2},
\label{ram:eqn13}
\end{equation}
\begin{align}
 &\sum_{n=0}^\infty \frac{q^{n(n-1)/2}}{(q; q)_n}\label{ram;eqn14}\\
&=\frac{(-q; q)_\infty}{(q; q)_\infty}\sum_{n=0}^\infty \sum_{j=-n}^n (-1)^{j} (1-q^{2n+1}) q^{j^2+n(n-1)/2}.
\nonumber
\end{align}
\section{The Sears $_4\phi_3$ transformation and Hecke-type series identities}
In this section we will prove Theorem~\ref{newliuthmg} using Theorem~\ref{mliuthmb} and the Sears $_4\phi_3$ transformation.
\begin{proof}
Setting $\gamma=h$ and $z=1$ in Theorem~\ref{mliuthmb}, we conclude that
\begin{align*}
&\frac{(\alpha q, \alpha ab/q; q)_\infty}
{(\alpha a, \alpha b; q)_\infty} {_3\phi_2} \left({{q/a, q/b, \beta} \atop { c, d}} ;  q, \frac{\alpha ab}{q} \right) \\
&=\sum_{n=0}^\infty \frac{(1-\alpha q^{2n}) (\alpha, q/a, q/b; q)_n (-\alpha ab/q)^n q^{n(n-1)/2}}
{(1-\alpha)(q, \alpha a, \alpha b; q)_n}
{_3\phi_2} \left({{q^{-n}, \alpha q^n, \beta} \atop {c, d}} ;  q, q \right).
\end{align*}
Applying (\ref{R:eqn12}) to the right-hand side of the above equation,  we complete the proof of Theorem~\ref{newliuthmg}.
\end{proof}
\section{Acknowledgments}
I am  grateful to the referee for many very helpful comments and suggestions.

\end{document}